\documentclass[11pt]{article}
\usepackage[utf8x]{inputenc}

\usepackage{amssymb,amsmath,amsthm}
\usepackage{graphicx,epsfig,color,hyperref}
\usepackage{enumerate}

\newtheorem{theorem}{Theorem}
\newtheorem{lemma}{Lemma}

\newtheorem{definition}{Definition}

\newtheorem{claimm}{Claim}{\itshape}{\rmfamily}

\newcommand{\es}{\hat{\oplus}}
\newcommand{\septhanks}{\ $^,$}

\title{The Structure of $W_4$-Immersion-Free Graphs\thanks{Emails of authors:
\href{mailto:remybelmonte@gmail.com}{{\sf
remybelmonte@gmail.com}}
\href{mailto:archontia.giannopoulou@gmail.com}{{\sf
archontia.giannopoulou@gmail.com}},
\href{mailto:daniello@ii.uib.no
}{{\sf daniello@ii.uib.no}},  
\href{mailto:sedthilk@thilikos.info}{{\sf
sedthilk@thilikos.info}}}}

\makeatletter\let\@fnsymbol\@alph\makeatother

\author{Rémy Belmonte\thanks{Dept. of Architecture and Architectural Engineering, Kyoto University, Japan}\septhanks\thanks{Supported by the ELC project (Grant-in-Aid for Scientific Research on Innovative
Areas, MEXT Japan)}
\and
Archontia Giannopoulou\thanks{Institute of Informatics, University of Warsaw, Warsaw, Poland.}\septhanks\thanks{Supported by the Warsaw Center of Mathematics and Computer Science.} 
\and 
Daniel Lokshtanov\thanks{Department of Informatics, University of Bergen, Norway.}\septhanks\thanks{Supported by BeHard  grant  under  the  recruitment  programme  of  the  Bergen  Research
Foundation.}
\and \\ Dimitrios
M. Thilikos\thanks{AlGCo project team, CNRS, LIRMM, Montpellier,
    France.}\septhanks\thanks{Department of Mathematics,
National and Kapodistrian University of Athens.}\septhanks\thanks{Co-financed by the European Union (European Social Fund ESF) and Greek national funds through the Operational Program “Education and Lifelong Learning” of the National Strategic Reference Framework (NSRF) - Research Funding Program: ARISTEIA II.}
}

\begin{document}

\date{\empty}
\maketitle

%
%
%
%
%
%
%
%
%
%
%

\begin{abstract}
\noindent We study the structure of graphs that do not contain the wheel on 5 vertices $W_4$ as an immersion, and show that these graphs can be constructed via 1, 2, and 3-edge-sums from subcubic graphs and graphs of bounded treewidth.
\end{abstract}


\section{Introduction}

A recurrent theme in structural graph theory is the study of specific properties that arise in graphs when excluding a fixed pattern.
The notion of appearing as a pattern gives rise to various graph containment relations.
Maybe the most famous example is the minor relation that has been widely studied, in particular since the fundamental results of Kuratowski and Wagner who proved that planar graphs are exactly those graphs that contain neither $K_5$ nor $K_{3,3}$ as a (topological) minor.
A graph $G$ contains a graph $H$ as a topological minor if $H$ can be obtained from $G$ by a sequence of vertex deletions, edge deletions and replacing internally vertex-disjoint paths by single edges.
Wagner also described the structure of the graphs that exclude $K_{5}$ as a minor: he proved that $K_{5}$-minor-free graphs can be constructed by ``gluing" together (using so-called clique-sums) planar graphs and a specific graph on $8$ vertices, called Wagner's graph. 

Wagner's theorem was later extended in the seminal Graph Minor series of papers by Robertson and Seymour (see e.g.~\cite{RobertsonS03a}), which culminated with the proof of Wagner's conjecture, i.e., that graphs are well-quasi-ordered under minors~\cite{RobertsonS04}, and ended with the proof of Nash-Williams' immersion conjecture, i.e., that the graphs are also well-quasi-ordered under immersions~\cite{RS10}. Other major results in graph minor theory include the (Strong) Structure Theorem~\cite{RobertsonS03a}, the Weak Structure Theorem~\cite{RobertsonS95b}, the Excluded Grid Theorem~\cite{RobertsonS86,RobertsonST94,KawarabayashiK12a}, as well as numerous others, e.g.,~\cite{SeymourT93-Grap,KawarabayashiRW11,DawarGK07}.
Moreover, the structural results of graph minor theory have deep algorithmic implications, one of the most significant examples being the existence of cubic time algorithms for the $k$-{\sc Disjoint Paths} and $H$-{\sc Minor Containment} problems~\cite{RobertsonS95b}. For more applications see, e.g.,~\cite{KawarabayashiW10,DemaineFHT05sube,AdlerKKLST11,KawarabayashiK08,BodlaenderFLPS09}.

However, while the structure of graphs that exclude a fixed graph $H$ as a minor has been extensively studied, the structure of graphs excluding a fixed graph $H$ as a topological minor or as an immersion has not received as much attention. While a general structure theorem for topological minor free graphs was very recently provided by Grohe and Marx~\cite{GM12}, 
finding an exact characterization of the graphs that exclude $K_5$ as a topological minor remains a notorious open problem.
Recently, Wollan gave a structure theorem for graphs excluding complete graphs as immersions~\cite{Wollan15}.
A graph $G$ contains a graph $H$ as a immersion if $H$ can be obtained from $G$ by a sequence of vertex deletions, edge deletions and replacing edge-disjoint paths by single edges.
Observe that if a graph $G$ contains a graph $H$ as a topological minor, then $G$ also contains $H$ as an immersion, as vertex-disjoint paths are also edge-disjoint.
In 2011, DeVos et al.~\cite{DDFMMS11} proved that if the minimum degree of a graph $G$ is at least $200t$ then $G$ contains the complete graph on $t$ vertices as an immersion.
In~\cite{FGTW08} Ferrara et al. provided a lower bound on the minimum degree of any graph $G$ in order to ensure that a given graph $H$ is contained in $G$ as an immersion.

A common drawback of such general results is that they do not provide sharp structural characterizations for concrete instantiations of the excluded graph $H$.
In the particular case of immersion, such structural results are only known when excluding both $K_{5}$ and $K_{3,3}$ as immersions\cite{GiannopoulouKT15}.
In this paper, we prove a structural characterization of the graphs that exclude $W_{4}$ as an immersion and show that they can be constructed from graphs that are either subcubic or have treewidth bounded by a constant. We denote by $W_4$ the wheel with~4 spokes, i.e., the graph obtained from a cycle on~4 vertices by adding a universal vertex.
The structure of graphs that exclude $W_4$ as a topological minor has been studied by Farr~\cite{Farr88}. He proved that these graphs can be constructed via clique-sums of order at most~3 from graphs of maximum degree at most~3. However, this characterization only applies to simple graphs. In our study we exclude $W_{4}$ as an immersion while allowing multiple edges. Robinson and Farr later extended this result by obtaining similar, albeit more complex, characterizations of graphs that exclude $W_6$ and $W_7$ as a topological minor~\cite{RF09a,RobinsonF14}.

As with the minor relation, many algorithmic results have also started appearing in terms of immersions.
In~\cite{GKMW11}, Grohe et al. gave a cubic time algorithm that decides whether a fixed graph $H$ immerses in any input graph $G$.
This algorithm, combined with the well-quasi-ordering of immersions~\cite{RS10}, implies that the membership of a graph in any graph class that is closed under taking immersions can be decided in cubic time. However, the construction of such an algorithm requires the ad-hoc knowledge of the finite set of excluded immersions that characterizes this graph class (which is called obstruction set). While no general way to compute an obstruction set is known, in~\cite{GiannopoulouSZ14}, Giannopoulou et al. provided sufficient conditions, under which the obstruction set of any graph class that is closed under taking immersions becomes effectively computable.
Another example of explicit construction of immersion obstruction sets is given by Belmonte et al.~\cite{BelmonteHKPT13}, where the set of immersion obstructions is given for graphs of carving-width~3.
Finally, for structural and algorithmic results on immersions in terms of colorings, see~\cite{KawarabayashiK12,Abu-KhzamL03,Lescure1988325,DKMO10}.

Our paper is organized as follows: in Section \ref{sec:preliminaries}, we give necessary definitions and previous results. In Section \ref{sec:invariance}, we show that containment of $W_4$ as an immersion is preserved under 1, 2 and~3-edge-sums. Then, in Section \ref{sec:main}, we provide our main result, i.e., a decomposition theorem for graphs excluding $W_4$ as an immersion. Finally, we conclude with remarks and open problems.

\section{Preliminaries}
\label{sec:preliminaries}
For undefined terminology and notation, we refer to the textbook of Diestel~\cite{Diestel}.
For every integer $n$, we let $[n]=\{1,2,\dots,n\}$.
All graphs we consider are finite, undirected, and without self-loops but may have  multiple edges. 
Given a graph $G$ we denote by $V(G)$ and $E(G)$
its {\em vertex} and {\em edge set} respectively.
Given a set $F\subseteq E(G)$ (resp. $S\subseteq V(G)$), we denote 
by $G\setminus F$ (resp. $G\setminus S$) the graph obtained from $G$ if we remove the edges in $F$ (resp. the vertices in $S$ along with their incident edges). 
We denote by ${\cal C}(G)$ the set of the {\em connected components} of $G$.
Given two vertices $u,v\in V(G)$, we also use the notation $G- v=G\setminus \{v\}$ and the notation $uv$ for the edge $\{u,v\}$. 
The {\em neighborhood} of a vertex $v\in V(G)$, denoted by $N_{G}(v)$, is the set of vertices in $G$ that are adjacent to $v$. We denote by $E_{G}(v)$ the set of the edges of $G$ that are incident with $v$.
The {\em degree} of a vertex $v\in V(G)$, denoted by $\deg_{G}(v)$, 
is the number of edges that are incident with it, that is, $\deg_{G}(v)=|E_{G}(v)|$. 
Notice that, as we are working with multigraphs, $|N_{G}(v)|\leq \deg_{G}(v)$.
The degree of a set $S$, denoted by $\partial(S)$, is the number of edges between $S$ and $V(G) \setminus S$, that is $|\{uv\in E(G) \mid u\in S \wedge v \not\in S\}|$.
Given two vertices $u$ and $v$ with $u\in N(v)$ we say that $u$ is an $i$-neighbor of $v$ if $E(G)$ contains exactly $i$ copies of the edge $\{u,v\}$. 
Let $P$ be a path and $v,u\in V(P)$. We denote by $P[v,u]$ the subpath of $P$ with endpoints $v$ and $u$.
The {\em maximum degree} of a graph $G$, denoted by $\Delta(H)$ is the maximum of the degrees of the vertices of $G$, that is, $\Delta(G)=\max_{v\in V(G)}\deg_{G}(v)$.

We denote by $W_{k-1}$ the {\em wheel} on $k$ vertices, that is, the graph obtained from the cycle of length $k-1$ after adding a new vertex and making it adjacent to all of its vertices. We call the new vertex {\em center} of the wheel.

\begin{definition}
An immersion of $H$ in $G$ is a function $\alpha$ with domain $V(H) \cup E(H)$, such that:
\begin{itemize}
\item $\alpha(v) \in V(G)$ for all $v \in V(H)$, and $\alpha(u)\neq\alpha(v)$ for all distinct $u,v \in V(H)$;
\item for each edge $e$ of $H$, 
$\alpha(e)$ is a path of $G$ with ends $\alpha(u), \alpha(v)$;
\item for all distinct $e,f \in E(H), E(\alpha(e) \cap \alpha(f))=\emptyset$.
\end{itemize}
\end{definition}

\noindent We call the image of every such function $\alpha$ in $G$ {\em model} of the graph $H$ in $G$ and the vertices of the set $\alpha(V(H))$ {\em branch} vertices of $\alpha$.

An {\em edge cut} in a graph $G$ is a non-empty set $F$ of edges that belong to the same connected component of $G$ and such that $G \setminus F$ has more connected components than $G$.
If $G \setminus F$ has one more connected component than $G$ and no proper subset of $F$ is an edge cut of $G$, then we say that $F$ is a {\em minimal} edge cut.
Given a vertex set $S$ such that $G[S]$ and $G \setminus S$ are connected, we denote by $(S,G \setminus S)$ the cut between $S$ and $G \setminus S$.
Let $F$ be an edge cut of a graph $G$ and let $G$ be the connected component of $G$ containing the edges of $F$. We say that $F$ is an {\em internal} edge cut if it is minimal and both connected components of $G \setminus F$ contain
at least 2 vertices.
An edge cut is also called {\em $i$-edge cut} if it has order at most $i$.

\begin{definition}
Let $G$, $G_1$, and $G_2$ be graphs. Let $t \geq 1$ be a positive integer. The graph $G$ is a $t$-edge-sum of $G_1$ and $G_2$ if the following holds.
There exist vertices $v_i \in V(G_i)$ such that $|E_{G_{i}}(v_i)|=t$ for $i\in[2]$ and a bijection $\pi: E_{G_{1}}(v_{1}) \rightarrow E_{G_{2}}(v_2)$ such that $G$ is obtained from $(G_{1} - v_{1}) \cup (G_{2} - v_{2})$ by adding an edge from $x\in V(G_1)-v_1$ to $y\in V(G_2)-v_2$ 
for every pair of edges $e_{1}$ and $e_{2}$ such that $e_{1}=xv_{1}$, $e_{2}=yv_{2}$, and $e_{2} = \pi(e_{1})$.
We say that the edge-sum is internal if both $G_1$ and $G_2$ contain at least 2 vertices and denote the internal $t$-edge-sum of $G_1$ and $G_2$ by $G_1 \hat{\oplus}_t G_2$.
\end{definition}

Note that if $G$ is the $t$-edge-sum of graphs $G_1$ and $G_2$ for some $t \geq 0$, then the set of edges $\{\{u,v\} \in E(G) \mid u \in V(G_1), v \in V(G_2)\}$ forms a minimal edge cut of $G$ of order $t$.\\

Let $r$ be a positive integer. The {\em $(r,r)$-grid} is the graph with vertex set $\{(i,j)\mid i,j\in [r]\}$ and edge set $\{\{(i,j),(i',j')\}\mid |i-i'| + |j-j'|=1\}$.
The {\em (elementary) wall} of height $r$ is the graph $W_r$ with vertex set $V(W_r) = \{(i,j) \mid i\in [r+1], j \in [2r+2]\}$ in which we make two vertices $(i,j)$ and $(i',j')$ adjacent if and only if either $i=i'$ and $j' \in \{j-1,j+1\}$ or $j'=j$ and $i'=i+(-1)^{i+j}$, and then remove all vertices of degree~1; see Figure~\ref{f-wall} for some examples.
The vertices of this vertex set are called {\em original} vertices of the wall.
A {\em subdivided wall} of height $r$ is the graph obtained from $W_{r}$ after replacing some of its edges by internally vertex-disjoint paths.

\begin{figure}[h]
\begin{center}
\includegraphics[scale=0.9]{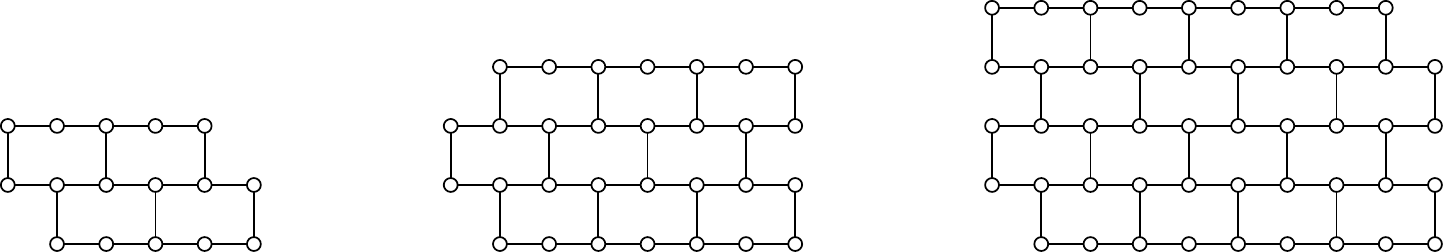} 
\caption{Elementary walls of height 2, 3, and 4.}\label{f-wall}
\end{center}
\end{figure}

Let $r$ be a positive integer and notice that the wall of height $r$ is contained in the $((2r+2)\times(2r+2))$-grid as a subgraph. This implies that any graph containing the $((2r+2)\times (2r+2))$-grid as a minor also contains the wall of height $r$ as a minor.
Furthermore, from a folklore result, for any simple graph $H$ with $\Delta(H)\leq 3$ it holds that $H$ is a minor of a graph $G$ if and only if $H$ is a topological minor of $G$.

\begin{theorem}\cite{LeafS12}
\label{thm:big-wallgen}
Let $G$ and $H$ be two graphs, where $H$ is connected and simple, not a tree, and has $h$ vertices. Let also $g$ be a positive integer. 
If $G$ has treewidth greater than $3(8h(h-2)(2g+h)(2g+1))^{|E(H)|-|V(H)|}+\frac{3h}{2}$ then $G$ contains either the $(g\times g)$-grid or $H$ as a minor.
\end{theorem}

 Theorem~\ref{thm:big-wallgen}, in the case where $g=2r+2$ and $H$ is the wall of height $r$, can be restated
as the well known fact that large treewidth ensures the existence of a large wall as a topological minor:

\begin{theorem}\cite{LeafS12}
\label{thm:big-wall}
Let $G$ be a graph and $r\geq 2$ be an integer. If the treewidth of $G$ is greater than $2^{18r^{2}\log r}$ then $G$ contains the wall of height $r$ as a topological minor.
\end{theorem}

We would like to note here that, regarding the dependence between the treewidth of a graph $G$ and the height of the wall that it contains as a topological minor, very recently, the following theorem was shown:

\begin{theorem}\cite{Chuzhoy15}
\label{thm:chuzhoywall}
There exists a function $f$ such that for any integer $r\geq 1$, any graph of treewidth at least $f(r)$ contains the wall of height $r$ as a topological minor, 
where $f(r)=O(r^{36}\text{poly}\log{r})$.
\end{theorem}

However, even though Theorem~\ref{thm:chuzhoywall} proves a tighter dependence between the treewidth of a graph and the height of the contained wall, we will use 
Theorem~\ref{thm:big-wall} instead of Theorem~\ref{thm:chuzhoywall}, as it allows us to extract specific constants for fixed values of $r$.

\section{Invariance of $W_4$ containment under small edge-sums}
\label{sec:invariance}

In this section, we show that immersion of $W_{4}$ is completely preserved under edge-sums of order at most~3, i.e., that $W_{4}$ immerses in a graph $G$ if and only if it immerses in at least one of the graphs obtained by decomposing $G$ along edge-sums. Theorem \ref{thm:immersion-preserved} will be necessary in Section~\ref{sec:main} to ensure that our decomposition does not change whether the graphs considered contain $W_4$ as an immersion or not. We first prove the following general lemma.

\begin{lemma}
\label{lem:immersion-preserved}
If $G$, $G_{1}$, and $G_{2}$ are graphs such that $G=G_{1}\es_{t}G_{2}$, $t\in [3]$, then both $G_{1}$ and $G_{2}$ are immersed in $G$.
\end{lemma}
  
\begin{proof} 
Notice that it is enough to  prove that $G_{1}$ is an immersion of $G$. Let $v_1$ and $v_2$ denote the unique vertex of $V(G_1) \setminus V(G)$ and $V(G_2) \setminus V(G)$ respectively. In the case where $G=G_{1}\es_{1}G_{2}$, let $u_{i}$ be the unique neighbor of $v_{i}$ in $G_{i}$, $i\in [2]$. Then 
the function $\{(v,v)\mid v\in V(G_{1}-v_{1})\}\cup \{(v_{1},u_{2})\}$ is an isomorphism from $G_{1}$ to the graph 
$G\setminus V(G_{2}-u_{2})$ (by the definition of the edge-sum $u_{2}u_{1}\in E(G)$) which is a subgraph of $G$. Therefore, $G_{1}\subseteq G$ and thus $G_{1}$ also immerses in $G$.

We now assume that $G=G_{1}\es_{t}G_{2}$, $t=2,3$. Let $e_{j}$, $j\in [|E_{G_{1}}(v_{1})|]$, be the edges of $E_{G_{1}}(v_{1})$ and let $u_{j}$ be the (not necessarily distinct) endpoints of the edges $e_{j}$, $j\in [|E_{G_{1}}(v_{1})|]$, in $G_{1}-v_{1}$.
Notice that in both cases, in order to obtain $G_{1}$ as an immersion of $G$, it is enough to find a vertex $u$ in $V(G)\setminus V(G_{1})$ and for each edge $e_{j}$ of $E_{G_{1}}(v_{1})$ find a path $P_{j}$ from $u$ to $e_{j}$ in $E(G)\setminus E(G_{1})$ such that these paths are edge-disjoint.
In what follows we find such vertex and paths. We distinguish the following cases.\\

\noindent {\em Case 1.}  $N_{G_{2}}(v_{2})=\{y\}$. Then, by the definition of the edge-sum, $G$ contains the edges $ye_{j}^{1}$, $j\in [|E_{G_{1}}(v_{1})|]$. 
Notice that neither the vertex $y$ belongs to $V(G_{1})$ nor the edges $yu_{j}^{1}$, $j\in [|E_{G_{1}}(v_{1})|]$, belong to $E(G_{1})$ and therefore the claim 
holds for $u=y$.\\

\noindent {\em Case 2.} $N_{G_{2}}(v_{2})=\{x,y\}$. First notice that in the case where $G=G_{1}\es_{3}G_{2}$ one of the $x,y$, say $x$, is a 
$2$-neighbor of $v_{2}$. As the edge-sum is internal, the set $E=E(G)\setminus (E(G_{1})\cup E(G_{2}))$ of edges created after the edge-sum is a minimal separator of $G$.
Without loss of generality let $yu_{1}^{1}$, $xu_{2}^{1}$, and (in the case where $G=G_{1}\es_{3}G_{2}$) $xu_{3}^{1}$ be its edges. 
By the minimality of the separator $E$, $G_{2}-v_{2}$ is connected. Therefore there exists a $(x,y)$-path $P$ in $G_{2}-v_{2}$. Observe that the path 
$P\cup \{yu_{1}^{1}\}$, the path consisting only of the edge $xu_{2}^{1}$ and (in the case where $G=G_{1}\es_{3}G_{2}$) the path consisting 
only of the edge $xu_{3}^{1}$ are edge-disjoint paths who do not have any edge from $E(G_{1})$ and share $x$ as a common endpoint. Then
the claim holds for $u=x$.\\

\noindent {\em Case 3.} $N_{G_{2}}(v_{2})=\{x,y,z\}$. In this case, it holds that $G=G_{1}\es_{3}G_{2}$. As above, consider the set 
$E=E(G)\setminus (E(G_{1})\cup E(G_{2}))$ of the edges created by the edge-sum and 
without loss of generality, let $E=\{xu_{1},yu_{2},zu_{3}\}$.
Since $E$ is a minimal separator, the graph $G_{2}-v_{2}$ is connected.
Therefore, there are a $(x,y)$-path $P$ and a $(y,z)$-path $Q$ in $G_{2}-v_{2}$. Let $z'$ be the vertex in $V(P)\cap V(Q)$ such that $V(Q[z,z'])\cap V(P)=\{z'\}$
and consider the paths $Q[z,z']$, $P[x,z']$, and $P[z',y]$ (in the case where $z'=y$ the path $P[z',y]$ is the graph consisting of only one vertex).
Observe that these graphs are edge-disjoint. Therefore the paths $P[x,z']\cup \{xu_{1}\}$, $P[y,z']\cup\{yu_{2}\}$, and $Q[z,z']\cup\{zu_{3}\}$ are edge-disjoint,
do not contain any edge from $E(G_{1})$, and share the vertex $z'$ as an endpoint. Thus, the claim holds for $u=z'$.
It then follows that $G_{1}$ is an immersion of $G$ and this completes the proof of the lemma.
\end{proof}

\begin{theorem}
\label{thm:immersion-preserved}
Let $G$, $G_{1}$, and $G_{2}$ be graphs such that $G = G_1 \hat\oplus_t G_2$, with $t \in [3]$. 
Then, $G$ contains $W_{4}$ as an immersion if and only if $G_{1}$ or $G_{2}$ does as well.
\end{theorem}

\begin{proof}
If $G_{1}$ or $G_{2}$ contains $W_{4}$ as an immersion, then $G$ does as well due to Lemma \ref{lem:immersion-preserved}.
It remains to prove the converse direction.

Let $\alpha$ be an immersion of $W_4$ in $G$. We first prove that either $|\alpha(V(W_4)) \cap (V(G_1)-v_1)| \geq 4$, 
or $|\alpha(V(W_4)) \cap (V(G_2)-v_2)| \geq 4$. Indeed, this is due to the fact that any cut $(S,G\setminus S)$ of $W_4$ 
with $|S|=3$ has order at least~4, whereas the cut $F=E(G)\setminus (E(G_{1})\cup E(G_{2}))$ in $G$ between 
$V(G_1)-v_1$ and $V(G_2)-v_2$ has order at most~3. 
Moreover, the same argument implies that the image of the center of $W_{4}$, that is, the unique vertex of degree~4 of $W_4$, say $x_{0}$, belongs to 
the connected component of $G-F$ that contains at least~4 of the branch vertices of the immersion $\alpha$. Let us assume without loss of generality that $x_0 \in V(G_1)-v_1$.

Assume first that $\alpha(V(W_4)) \cap (V(G_1)-v_1) = 5$. 
If for every edge $e$ of $W_{4}$ it holds that $\alpha(e)\cap V(G_{2}-v_{2})=\emptyset$, then clearly $\alpha$ is an immersion of $W_4$ in $G_1-v_1$, and therefore in $G_1$. Moreover, it is easy to observe that there cannot be two distinct 
edges $e,e'$ of $W_4$ whose image path in $G$ contains vertices of $G_2-v_2$, since each such path must contain at least~2 edges of $F$, and $|F| \leq 3$. 
Hence we may assume that there exists a unique edge $e$ with $\alpha(e) \cap V(G_2-v_2) \neq \emptyset$. Note that $\alpha(e)$ must intersect the 
cut $F$ in an even number of edges, since otherwise the path would end in $G_2-v_2$, contradicting our assumption that all branch vertices of 
$\alpha$ lie in $G_1-v_1$. Let $P$ be the maximum subpath of $\alpha(e)$ such that $E(P')\cap E(G_{1}-v_{1})=\emptyset$.
Notice that the first and the last edge of such a path are edges of $F$. Let $u_{1}$ and $u_{2}$ be the endpoints of $P$. This implies that we may obtain an immersion $\alpha'$ of $W_{4}$ in $G_{1}$ by replacing in $\alpha$ the path $P$ by the path $u_{1}v_{1}u_{2}$.

Now, we assume that $\alpha(V(W_4)) \cap (V(G_1)-v_1) = 4$, and denote by $x$ the unique branch vertex of $\alpha$ lying in $V(G_2-v_2)$. 
We claim that it is possible to create an immersion function $\alpha'$ of $W_4$ in $G_1$ by replacing the vertex $x$ in $\alpha$ with $v_1$. 
To show this, we apply the following operations to $G$: let $P_1,P_2,P_3$ be the paths of $\alpha$ whose associated edges in $W_4$ are incident 
with $\alpha^{-1}(x_4)$, and let $P'_1,P'_2,P'_3$ be 
the subpaths of $P_1,P_2$, and $P_3$ 
that do not contain edges of $G_1-v_1$. The paths $P'_1,P'_2,P'_3$ are easily observed to be edge-disjoint, and 
therefore we may lift the edges in each of these paths. We complete the construction by deleting the vertices in $V(G_2)-\{v_2,x_4\}$. The graph obtained 
from this construction is readily observed to be isomorphic to $G_1$ by mapping every vertex of $G_1-v_1$ to itself, and $v_1$ to $x$. Therefore 
$W_4$ immerses in $G_1$. This concludes the proof of the theorem.
\end{proof}

\section{Structure of graphs excluding $W_4$ as an immersion}
\label{sec:main}

In this section, we prove the main result of our paper, namely we provide a structure theorem for graphs that exclude $W_4$ as an immersion. We first provide a technical lemma that will be crucial for the proof of Theorem \ref{thm:main}.

\begin{lemma}
\label{lem:big-lemma}
There exists a function $f$ such that for every integer $r \geq 60000$ and every graph $G$
that does not contain $W_4$ as an immersion, has no internal 3-edge cut, and has a vertex $u$ with $d(u) \geq 4$,
if $tw(G) \geq f(r)$, then there exist vertex sets $S_1,\ldots,S_r$, $Z=\{z_1,\ldots,z_r\}$, and $X$ of $G$, that satisfy the following properties:
\begin{enumerate}[(i)]
\item $z_i \in S_{i}, \forall i \in \{1,\ldots,r\}$;
\item $z_i \not\in S_j, \forall i \neq j \in \{1,\ldots,r\}$;
\item $u \in \bigcap_{i \in \{1,\ldots,r\}}S_i$;
\item $\partial(S_i) \leq 6$;
\item $G[S_i]$ is connected, $\forall i \in \{1,\ldots,r\}$;
\item $X \cap S_i = \emptyset, \forall i \in \{1,\ldots,r\}$;
\item For every $Z' \subseteq Z$ such that $|Z'| \geq 7$, there is a 7-flow from $Z'$ to $X$;
\end{enumerate}
\end{lemma}

\begin{proof}
Assume that $G$ has treewidth at least $2^{18(6r)^{2}\log (6r)}$, Then, from Theorem~\ref{thm:big-wall}, $G-u$ contains an elementary wall of height $6r$ 
as a topological minor and hence a subdivided wall $W$ of height $6r$ as a subgraph.
We define the cycles $C_1,\ldots,C_{6r}$ as the ones formed by the (original) vertices $w_{5+20p,3+2q}$ to 
$w_{11+20p,3+2q}$ and $w_{11+20p,4+2q}$ to $w_{5+20p,4+2q}$
and the internally vertex-disjoint paths that join them on the wall $W$, for every $p,q \in \{0,\ldots,\lceil\sqrt{6r}\rceil-1\}$.
Observe that $C_1,\ldots,C_{6r}$ is a set of vertex disjoint cycles containing at least~14 vertices in $G-u$.
For every $i \in [6r]$, we denote by $G_{C_i}$ the graph obtained from $G$ by removing the edges of $C_i$ and adding a vertex $v_i$ adjacent exactly to the vertices of $C_i$.
Since $W_4$ does not immerse in $G$, there exists an edge cut $F_i$ of order at most~3 that separates $u$ and $v_i$, 
as otherwise we would be able to find 4 edge-disjoint paths from $u$ to the vertices of the cycle $C_{i}$ and thus, an immersion model of $W_{4}$ on $G$.
Moreover, since both $u$ and $v_i$ have degree at least~4, this edge cut is internal.
We now define the set $T_i$, for every $i \in [6r]$, as the set of vertices that lie in the same connected component of $G_{C_i}-F_i$ as $u$.
\begin{claimm}\label{claim:intersection}
For every $i \in \{1,\ldots,6r\}$, $1 \leq |T_i \cap V(C_i)| \leq 3$.
\end{claimm}

\noindent
\textit{Proof of Claim \ref{claim:intersection}.}
The fact that $|T_i \cap V(C_i)| \geq 1$ follows from the observation that if $T_i \cap V(C_i) = \emptyset$, then the cut $F_i$ is not only a cut in $G_{C_i}$, but also in $G$, which contradicts the assumption that $G$ is internally~4-edge-connected.
On the other hand, observe that for every vertex $w \in V(C_i) \cap T_i$, the edge $v_{i}w$ must belong to the cut $F_i$, as otherwise there is a path joining 
$u$ and $v_{i}$ in $G-F_{i}$, a contradiction. 
Therefore no more than~3 vertices of $C_i$ may lie in $T_i$, which concludes the proof of the claim.
\hfill $\diamond$\\

We now define the set $Z=\{z_1,\ldots,z_{6r}\}$: for every $i \in [6r]$, we choose arbitrarily one vertex of $T_i \cap V(C_i)$ to be the vertex $z_i$. The existence of the vertices $z_i$ follows from Claim~\ref{claim:intersection}. Observe that, by construction of $Z$, it holds that $z_i \in T_i, \forall i \in [6r]$, i.e., the sets $T_i$ satisfy property (i).

Observe that, by construction, $G[T_i]$ is connected.
Moreover, the only edges of $G$ that are not edges of $G_{C_i}$ are the edges of the cycle $C_i$.
Thus, the only edges in the cut $(T_i, G\setminus T_i)$ of $G$ that are not edges of the cut $F_i$ in $G_{C_i}$ are the edges of $C_i$ incident with the vertices of $T_i \cap V(C_i)$.
Furthermore, for every vertex $w$ of $T_i \cap V(C_i)$, the edge $wv_i$ belongs to the cut $F_i$ in $G_{C_i}$, but not to the cut $(T_i, G\setminus T_i)$ in $G$.
Hence, the number of edges of the cut $(T_i, G\setminus T_i)$ in $G$ is at most $|F_i| + 2|T_i \cap V(C_i)| - |T_i \cap V(C_i)|$.
Since $F_i$ and $T_i \cap V(C_i)$ both have order at most~3, it follows that the cut $(T_i, G\setminus T_i)$ in $G$ has order at most~6.
We have therefore proved that properties (iii)-(v) hold for the sets $T_i, i \in [6r]$.

We may now define the set $X$. We first start with the set of (original) vertices $w_{p,q}$ of the wall, with $6r+1 - 36(\lceil\sqrt{6r}\rceil+1) \leq p \leq 6r+1 + 36(\lceil\sqrt{6r}\rceil+1)$ and $6r+1 - (\lceil\sqrt{6r}\rceil+1) \leq q \leq 6r+1 - 73(\lceil\sqrt{6r}\rceil+1)$. This set, denoted $X_0$, contains at least $72(\sqrt{6r}+1)^2 \geq 72r+73$ original vertices of the wall, due to $r \geq 1$. We now need the following:

\begin{claimm}\label{claim:overlap}
For every $i \in [6r]$, $|X_0 \cap T_i| \leq 72$.
\end{claimm}

\noindent
\textit{Proof of Claim \ref{claim:overlap}.}
We prove the claim by showing that for every $i \in [6r]$ and every subset $X'_0$ of $X_0$ that contains at least~73 vertices, there are~7 disjoint paths from vertices of $C_i \setminus T_i$ to vertices of $X'_0$ in $G$. Together with property (iv), this will imply validity of Claim \ref{claim:overlap}.
Consider a subset $X'_0$ of $X_0$ that contains least~73 original vertices of the wall. Observe that there must be 13 vertices that lie on the same horizontal path, or 7 vertices that lie on different horizontal paths. From there, taking into account the dimensions of the wall and the position of the vertices of $C_i$ and $X_0$, it is easy to observe that there always exist vertices $y_1,\ldots,y_7$ in $C_i \setminus T_i$ and $x_1,\ldots,x_7$ in $X'_0$ such that there are~7 disjoint paths between $y_1,\ldots,y_7$ and $x_1,\ldots,x_7$.
\hfill$\diamond$\\

Therefore, the set $X_0 \cap \bigcup T_i$ contains at most $72r$ vertices, which implies that there exists a subset $X$ of $X_0$ containing at least $73$ vertices such that $X \cap T_i = \emptyset$ for every $i \in [6r]$. This proves property (vi) for the sets $T_i, i \in [6r]$.
The validity of property (vii) follows from arguments similar to those given in the proof of Claim \ref{claim:overlap}.

Finally, we show  how to select sets $S_1,\ldots,S_r$ among $T_1,\ldots,T_{6r}$ so that property (ii) holds, namely that for every $1 \leq i \neq j \leq r, z_i \not\in S_j$.
In order to find such sets, we proceed as follows: let $H$ be a directed graph such that $V(H)=\{T_1,\ldots,T_{6r}\}$, and $(T_i,T_j)$ is an arc of $H$ if and only if $z_i \in S_j$. We now claim that vertices of $H$ have indegree at most~6. This is shown by combining properties (iv), (vi), and (vii). Assume for contradiction that there is a vertex in $H$ having indegree at least~7, then there exist distinct indices $i_1,\ldots,i_7$ and $j$ such that $z_{i_1},\ldots,z_{i_7} \in S_j$. However, we know that there exist~7 disjoint paths from $\{z_{i_1},\ldots,z_{i_7}\}$ to $X$ by property (vii). Together with property (vi), we obtain a contradiction with property (iv). Therefore, we conclude that the directed graph $H$ has maximum indegree at most~6. Thus, $|E(H)| \leq 36r$, which implies that the average degree of $H$ is at most~6. Hence, $H$ is~6-degenerate and thus contains an independent set of size at least $\frac{|V(H)|}{6} = r$. The vertices of such an independent set correspond to sets $T_{i_1},\ldots,T_{i_
r}$ such that, for every $1 \leq p \neq q \leq r$, $z_{i_p} \not\in T_{i_q}$. Therefore, we choose $S_p := T_{i_p}$ for every $p \in [r]$ and observe that the set $S_1,\ldots,S_r$ as defined indeed satisfy property (ii).

Finally, since every set $T_i$ satisfies properties (i) and (iii)-(vi), and for every $j \in [r]$ there exists $i \in [6r]$ such that $S_j=T_i$, we obtain that the sets $S_i$ satisfy these properties as well.
This concludes the proof of the lemma.
\end{proof}

Lemma \ref{lem:big-lemma} essentially states that large treewidth yields a large number of vertex disjoint cycles that are highly connected to each other, and an additional disjoint set that is highly connected to these cycles. However, this, together with the assumption that $W_4$ does not immerse in $G$, implies that there cannot be a large flow between a vertex of degree at least~4 and one of the cycles. We will combine this fact with the notion of important separators to obtain Lemma~\ref{lem:bounded-treewidth}.

\begin{definition}
Let $X,Y \subseteq V(G)$ be vertices, $S \subseteq E(G)$ be an $(X,Y)$-separator, and let $R$ be the set of vertices reachable from $X$ in $G \setminus S$.
We say that $S$ is an important $(X,Y)$-separator if it is inclusion-wise minimal and there is no $(X,Y)$-separator $S'$ with $|S'| \leq |S|$ such that $R' \subset R$, where $R'$ is the set of vertices reachable from $X$ in $G \setminus S'$.
\end{definition}

\begin{theorem}\cite{Marx06,CLL07}
\label{thm:imp-sep}
Let $X,Y \subseteq V(G)$ be two sets of vertices in graph $G$, let $k \geq 0$ be an integer, and let $S_k$ be the set of all $(X,Y)$-important separators of size at most $k$. Then $|S_k| \leq 4^k$ and $S_k$ can be constructed in time $|S_k| \cdot n^{O(1)}$.
\end{theorem}

Theorem \ref{thm:imp-sep} states that the number of important separators of a certain size is bounded. The next lemma combines this fact with Lemma \ref{lem:big-lemma}.

\begin{lemma}
\label{lem:bounded-treewidth}
Let $G$ be a graph such that $G$ does not contain $W_4$ as an immersion, has no internal 3-edge cut and has a vertex $u$ with $d(u) \geq 4$.
Then the treewidth of $G$ is upper bounded by a constant.
\end{lemma}

\begin{proof}
If $G$ has treewidth at least $2^{18(6r)^{2}\log (6r)}$ for $r \geq 60000$, then there exist sets $Z=\{z_1,\ldots,z_r\}$, $S_1,\ldots,S_r$ and $X$ that satisfy the properties of Lemma~\ref{lem:big-lemma}.
Recall that $F$ is an important $(u,X)$-separator if there is no $(u,X)$-separator $F'$ such that $|F'| \leq |F|$ and the connected component of $G-F$ that contains $u$ is properly contained in the connected component of $G-F'$ that contains $u$.
Additionally, observe that for every set $S_i$, there is an important separator $F$ or order at most~6 such that $S_i$ lies in the same connected component as $\{u\}$ in $G-F$. 
Moreover, for any cut $F$ of order at most~6 such that $S_i$ is contained in the same connected component as $u$ in $G-F$, there cannot be~7 disjoint paths from $u$ to $X$ through $F$.
Combined with property (vii) of Lemma \ref{lem:big-lemma} and the fact that every set $S_i$ contains a vertex $z_i$, this implies that for every important separator $F$, there are at most~6 sets $S_{i_1},\ldots,S_{i_p}, p \leq 6$, that are contained in the same connected component as $u$ in $G-F$.
However, Theorem~\ref{thm:imp-sep} ensures that there are at most $4^6$ important $(X,\{u\})$-separators of size at most~6 in $G$.
Therefore, if $r \geq 60000 > 6 \cdot 4^6$, there is a set $S_i$ such that the cut $(S_i,G-S_i)$ has order at least~7. Thus, we conclude that either $G$ has an internal edge cut of order at most~3, or it has no vertex of degree at least~4, or it contains $W_4$ as an immersion. Hence the lemma holds.
\end{proof}

We are now ready to prove the main theorem of our paper.

\begin{theorem}
\label{thm:main}
Let $G$ be a graph that does not contain $W_4$ as an immersion. Then the prime graphs of a decomposition of $G$ via $i$-edge-sums, $i\in [3]$, are either subcubic graphs, or have treewidth upper bounded by a constant.
\end{theorem}

\begin{proof}
Let us consider a decomposition of $G$ via $i$-edge-sums, $i \in [3]$, and let $H$ be a prime graph of such a decomposition. Note first that, since $G$ does not contain $W_4$ as an immersion, then $H$ does not contain it either, due to Theorem~\ref{thm:immersion-preserved}. Now, assume that $H$ is not subcubic. Then there is a vertex $u$ of degree at least~4 in $H$. Moreover, it is clear from Theorem~\ref{thm:immersion-preserved} that $H$ is internally~4 edge-connected. Hence, we may apply Lemma~\ref{lem:bounded-treewidth} and conclude that $H$ has treewidth at most $2^{2^{13}\cdot 3^{6}\cdot 5^{8}\cdot\log (2^{6}\cdot 3^{2}\cdot 5^{4})}$. Thus, the theorem holds.
\end{proof}

We conclude this section by noting that Theorem \ref{thm:main} is in a sense tight: indeed, both the fact that we decompose along edge-sums of order at most~3 and the requirement that a unique vertex of degree at least~4 is sufficient to enforce small treewidth are necessary. The fact that decomposing along internal 3-edge-sums is necessary can be seen from the fact that there are internally~3 edge-connected graphs that have vertices of degree at least~4 and yet do not contain $W_4$ as an immersion, e.g., a cycle where every edge is doubled.

\section{Concluding remarks}

Following the proof of Theorem \ref{thm:main}, the first task is to improve the bound on the treewidth of internally~4 edge-connected graphs that exclude $W_4$ as an immersion and have a vertex of degree at least~4.
Our proof of Theorem \ref{thm:main} relies on the fact that large treewidth ensures the existence of a large number of vertex disjoint cycles that are highly connected to each other.
In order to obtain these cycles, we use the fact that graphs of large treewidth contain a large wall as a topological minor. However, the value of treewidth required to find a sufficiently large wall is currently enormous.
Avoiding to rely on the existence of a large wall would be an efficient way to drastically reduce the constants in Lemma \ref{lem:big-lemma} and Theorem~\ref{thm:main}.

Another question that we leave open is to prove a similar result for larger wheels, i.e.,~$W_k$ for $k \geq 5$. Providing a decomposition theorem for larger wheel seems to be a challenging task, as edge-sums no longer seem to be the proper way to proceed, since, as argued in Section~\ref{sec:main}, $k$ edge connectivity is necessary, but $W_k$-immersion is not preserved under edge-sums of order $k-1$, as seen in Figure~\ref{fig:ex1} and~\ref{fig:ex2}.

\begin{figure}[h]
\begin{center}
\includegraphics[scale=0.17]{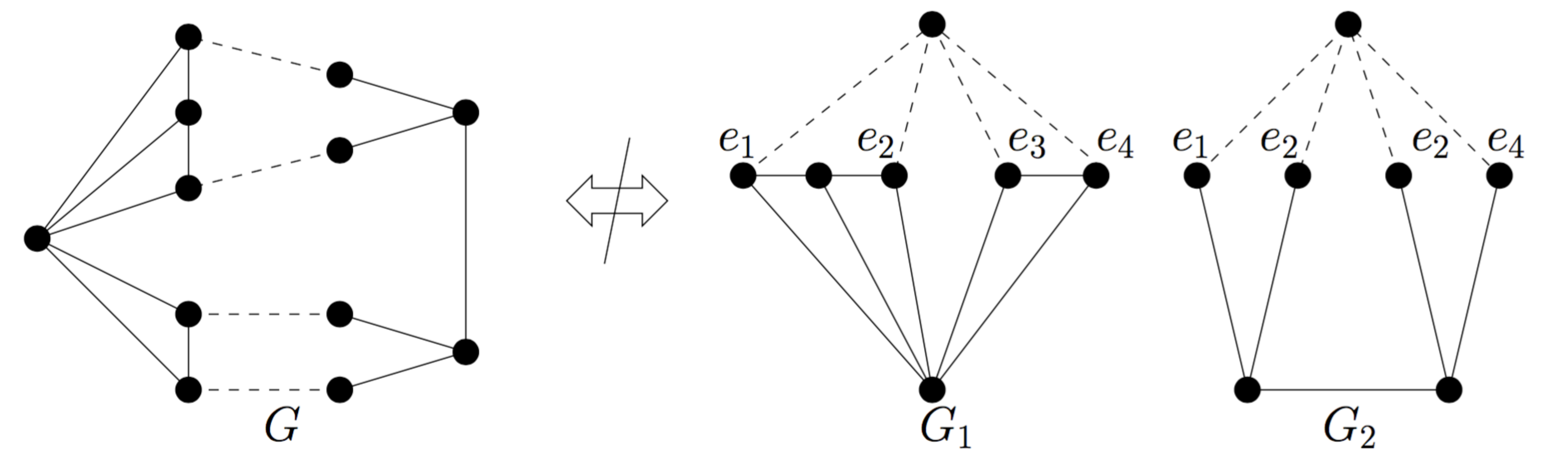} 
\caption{$G_{1}$ contains $W_{5}$ as an immersion but $G=G_{1}\es_{4}G_{2}$ does not.
The unique vertex in $G_i$ incident with dotted edges is the vertex $v_i$, and the edge-sum maps to each other edges of $G_1$ and $G_2$ with the same label.}
\label{fig:ex1}
\end{center}
\end{figure}

Decomposition theorems exist when small wheels are excluded as topological minors \cite{Farr88,RF09a,RobinsonF14}, however these results do not apply when excluding wheels as immersions, as in this case we must  consider multigraphs.
A similar important question is to characterize graphs excluding $K_5$ as an immersion.

Finally, note that the general algorithm to test immersion containment runs in cubic time for every fixed target graph $H$. We believe that Theorem \ref{thm:main} can be used to devise efficient algorithms to recognize graphs that exclude $W_4$ as an immersion.

\begin{figure}[h]
\begin{center}
\includegraphics[scale=0.23]{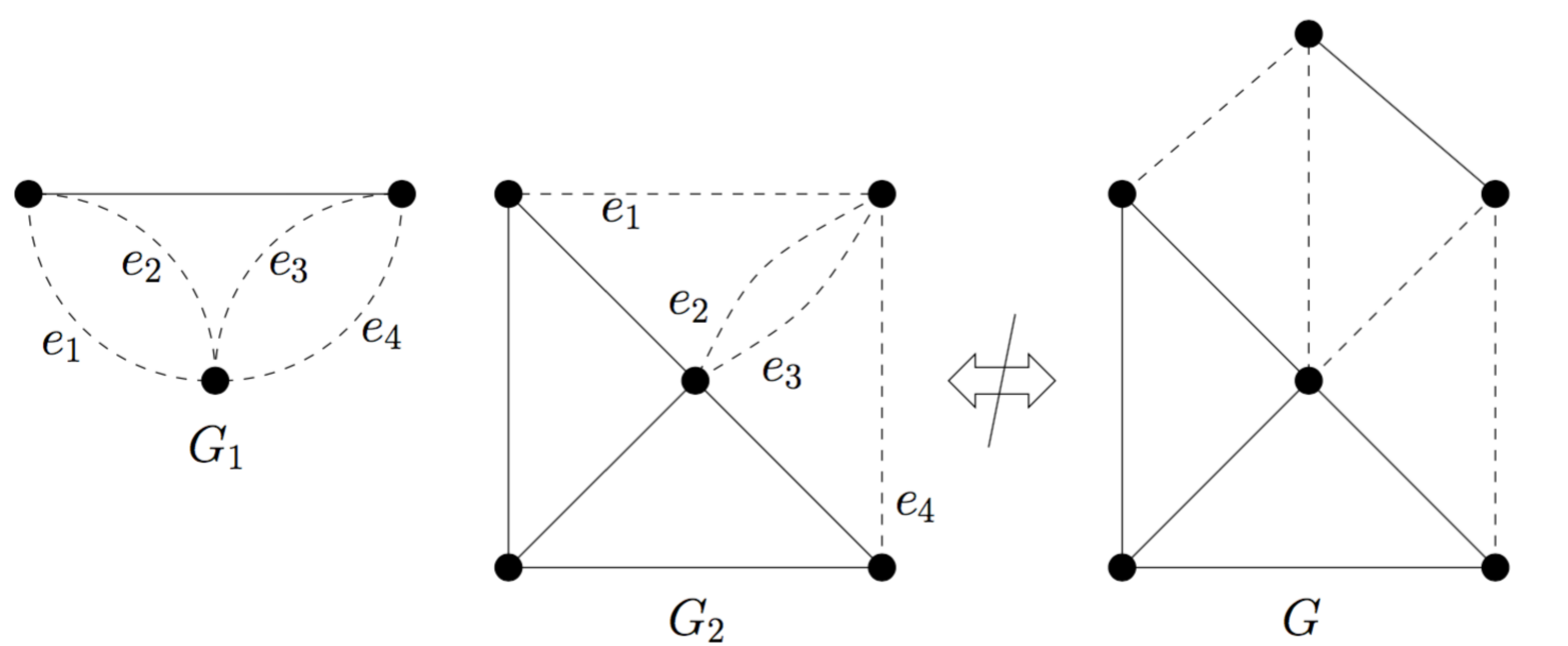} 
\caption{Neither $G_{1}$ nor $G_{2}$ contain $W_{4}$ as immersion but $G=G_{1}\es_{4}G_{2}$ does.
The unique vertex in $G_i$ incident with dotted edges is the vertex $v_i$, and the edge-sum maps to each other edges of $G_1$ and $G_2$ with the same label.}
\label{fig:ex2}
\end{center}
\end{figure}

%
%
\end{document}